\documentclass[leqno,12pt]{article} 
\usepackage{amsmath, amssymb}
\usepackage{amsthm} 
\usepackage{amsbsy}
\usepackage{amsfonts}
\usepackage{mathtools}
\usepackage{bbm}
\usepackage{comment}
\usepackage{color}

\theoremstyle{plain} 
\newtheorem{theorem}{\indent\sc Theorem}[section]
\newtheorem{lemma}[theorem]{\indent\sc Lemma}
\newtheorem{corollary}[theorem]{\indent\sc Corollary}

\theoremstyle{definition} 

\newtheorem{remark}[theorem]{\indent\sc Remark}

\def\T{\texttt{T}}

\def\det{{\rm det}}

\makeatletter
\def\address#1#2{\begingroup
\noindent\parbox[t]{7.8cm}{%
\small{\scshape\ignorespaces#1}\par\vskip1ex
\noindent\small{\itshape E-mail address}%
\/: #2\par\vskip4ex}\hfill%
\endgroup}%
\makeatother
\title{\uppercase{ON THE STRUCTURE OF LAPLACE CHARACTERISTIC POLYNOMIAL FOR CIRCULANT FOLIATION
}} 
\author{
\textsc{ Y.~S.~Kwon, A.~D.~Mednykh and I.~A.~Mednykh} 
}
\date{} 
\begin{document}

\maketitle

\footnote{
2010 \textit{Mathematics Subject Classification}.
Primary 05C30; Secondary 39A10.}
\footnote{
\textit{Key words and phrases}.
Spanning forest, rooted forest, circulant graph, $I$-graph, $Y$-graph, $H$-graph, Laplace matrix, Chebyshev polynomial.
}
\footnote{
The first and the third authors were supported by the Mathematical
Center in Akademgorodok, agreement no. 075-15-2019-1613 with the
Ministry of Science and Higher Education of the Russian Federation.
The second author was supported by the Basic Science Research
Program through the National Research Foundation of Korea (NRF)
funded by the Ministry of Education (2018R1D1A1B05048450).
}

\begin{abstract}
In this paper, we describe the structure of the Laplace characteristic polynomial $\chi_n(\lambda)$ for the infinite family of graphs $H_n=H_n(G_1,\,G_2,\ldots,G_m)$ obtained as a circulant foliation over a graph $H$ on $m$ vertices with fibers $G_1,\,G_2,\ldots,G_m.$ Each fiber $G_i=C_n(s_{i,1},\,s_{i,2},\ldots,s_{i,k_i})$ of this foliation is the circulant graph on $n$ vertices with jumps $s_{i,1},\,s_{i,2},\ldots,s_{i,k_i}.$ This family includes the family of generalized Petersen graphs, $I$-graphs, sandwiches of circulant graphs, discrete torus graphs and others. We show that the characteristic polynomial for such graphs can be decomposed into a finite product of algebraic functions evaluated at the roots of a linear combination of Chebyshev polynomials. Also, we prove that the characteristic polynomial can be represented in the form $\chi_n(\lambda)=p(\lambda)\,\chi_H(\lambda)a(n)^2,$ where $a(n)$ is a sequence of integer polynomials and $p(\lambda)$ is a prescribed integer polynomial. Moreover, we use the obtained results to produce analytic formulas for spectral graph invariants, such as the number of spanning trees and the number of spanning rooted forests.   
\end{abstract}

\section{Introduction} 

Let $G$ be a finite connected undirected graph. The complexity of the graph can be characterized through the number of edges or vertices, the number of spanning trees, rooted spanning forests or the Kirchhoff index. 


All the above mentioned characteristics can be expressed in terms of Laplacian eigenvalues of a graph or more precisely, through coefficients of its Laplace characteristic polynomials. The Laplace characteristic polynomial of a graph plays an important role in statistic physics, where the graphs with arbitrary large number of vertices are considered (\cite{Wu77}, \cite{JacSalSok}). In this case, the structure of the Laplace characteristic polynomial becomes quite complicated and the most interesting invariants are given by its asymptotic.

There are enormous number of papers devoted to counting of the number of spanning trees. 
See survey papers \cite{BoePro} and \cite{WeiFu}. On the other hand, little is known about analytic 
formulas for the number of spanning forests. One of the first results was obtained by O. Knill 
\cite{Knill} who proved that the number of rooted spanning forests in the complete graph $K_n$ 
on $n$ vertices is equal to $(n+1)^{n-1}.$ Explicit formulas for the number of rooted spanning 
forests in cycle, star, line graph and some other graphs were given in (\cite{Sz}, \cite{Knill}). 
The number of rooted forests in circulant graphs has been calculated in \cite{GrunMed}. 
Generalizations of these results for circulant foliation were done in \cite{GrunKwonMed}. The results 
of the mentioned papers can be easily obtained if we know the structure of the Laplace characteristic polynomial.

The aim of the present paper is to investigate the structure of the Laplace characteristic polynomial 
for circulant foliation over a given graph. The notion of circulant foliation was introduced in paper
\cite{KwonMedMed20}. We note that the family of such graphs is quite rich. It includes circulant 
graphs, generalized Petersen graphs, $I$-, $Y$-, $H$-graphs, discrete tori and others.

The structure of the paper is as follows. Some preliminary results and basic definitions are given in Section~\ref{basic}. In Section~\ref{circulant foliation} we define the notion of circulant foliation over a graph. 
In Section~\ref{counting}, we describe the structure of the Laplace characteristic polynomial $\chi_n(\lambda)$ 
of the graph $H_n=H_n(G_1,\,G_2,\ldots,G_m)$ obtained as a circulant foliation over a graph $H$ on $m$ vertices with fibers $G_1,\,G_2,\ldots,G_m.$ Each fiber $G_i=C_n(s_{i,1},\,s_{i,2},\ldots,s_{i,k_i})$ of this foliation is the circulant graph on $n$ vertices with jumps $s_{i,1},\,s_{i,2},\ldots,s_{i,k_i}.$ It will be shown that the characteristic polynomial for such graphs can be decomposed into a finite product of algebraic functions evaluated at the roots of a linear combination of Chebyshev polynomials. 
In Section~\ref{arithmetic}, we prove that the characteristic polynomial can be represented in the
form $\chi_n(\lambda)=p(\lambda)\,\chi_H(\lambda)a(n)^2,$ where $a(n)$ is a sequence of  integer polynomials and $p(\lambda)$ is a prescribed integer polynomial depending on the number of odd elements in the set of $s_{i,j}.$ 
In Section~\ref{application}, we applied the obtained results to find analytical formulas for the number of spanning  trees and the number of spanning forests in circulant foliation $H_n.$ The number of rooted $k$-colored spanning forests will also be obtained. 
In the last section, we illustrate the obtained results by a series of examples.

\section{Basic definitions and preliminary facts}\label{basic}

Consider a connected finite graph $G,$ allowed to have multiple
edges but without loops. We denote the vertex and edge set of $G$ by
$V(G)$ and $E(G),$ respectively. Given $u, v\in V(G),$ we set
$a_{uv}$ to be the number of edges between vertices $u$ and $v.$ The
matrix $A=A(G)=\{a_{uv}\}_{u, v\in V(G)}$ is called \textit{the
adjacency matrix} of the graph $G.$ The degree $d_{v}$ of a vertex
$v\in V(G)$ is defined by $d_v=\sum_{u\in V(G)}a_{uv}.$ Let $D=D(G)$
be the diagonal matrix indexed by the elements of $V(G)$ with
$d_{vv} = d_v.$ The matrix $L=L(G)=D(G)-A(G)$ is called \textit{the
Laplace matrix}, or simply \textit{Laplacian}, of the graph $G.$
\textit{The characteristic polynomial} of the Laplacian of $G$ is defined as the
polynomial $ \chi_G (\lambda) = \det(L(G) - \lambda I),$ where $I$ is the identity matrix of order equal to the number of vertices of $G.$

In this paper we prefer to deal with more general definition of
Laplacian. Let $X=\{x_v,\,v\in V(G)\}$ be the set of variables and
let $X(G)$ be the diagonal matrix indexed by the elements of $V(G)$
with diagonal elements $x_v.$ The \textit{generalized Laplace
matrix} of $G,$ denoted by $L(G,X),$ is given by $L(G,X)=X(G)-A(G).$
In the particular case $x_v=d_v,$ we have $L(G,X)=L(G).$


We call an $n\times n$ matrix \textit{circulant,} and denote it by $circ(a_0, a_1,\ldots,a_{n-1})$ if it is of the form
$$circ(a_0, a_1,\ldots, a_{n-1})=
\left(\begin{array}{ccccc}
a_0 & a_1 & a_2 & \ldots & a_{n-1} \\
a_{n-1} & a_0 & a_1 & \ldots & a_{n-2} \\
  & \vdots &   &  & \vdots \\
a_1 & a_2 & a_3 & \ldots & a_0\\
\end{array}\right).$$

Recall \cite{PJDav} that the eigenvalues of matrix $\texttt{C}=circ(a_0,a_1,\ldots,a_{n-1})$ 
are given by the following simple formulas $\lambda_j=p(\varepsilon^j_n),\,j=0,1,\ldots,n-1$
where $p(x)=a_0+a_1 x+\cdots+a_{n-1}x^{n-1}$ and $\varepsilon_n$ is an order $n$ 
primitive root of the unity. Moreover, the circulant matrix $\texttt{C}=p(\T_n ),$ where 
$\T_{n}=circ(0,1,0,\ldots,0)$ is the matrix representation of the shift operator
$\T_{n}:(x_0,x_1,\ldots,x_{n-2},x_{n-1})\rightarrow(x_1,x_2,\ldots, x_{n-1},x_0).$ 
We note that all $n\times n$ circulant matrices share the same set of eigenvectors. 
Hence, any set of $n\times n$ circulant matrices can be simultaneously diagonalizable.

Let $s_1,s_2,\ldots,s_k$ be integers such that $1\leq s_1<s_2<\cdots<s_k<\frac{n}{2}.$ 
The graph $C_n(s_1,s_2,\dots,s_k)$ with $n$ vertices $0,1,2,\dots,n-1$ is called 
\textit{circulant graph} if the vertex $i,0\leq i\leq n-1$ is adjacent to the vertices
$i\pm s_1, i\pm s_2,\dots,i\pm s_k$ (mod $n$). All vertices of the graph are of even 
degree $2k$. In this paper, we also allow the empty circulant graph $C_n(\emptyset)$ 
consisting of $n$ isolated vertices.

\section{Circulant foliation over a graph}\label{circulant foliation}

Let $H$ be a finite graph with vertices $v_1,v_2,\ldots,v_m,$
allowed to have multiple edges but without loops. Denote by
$a_{i\,j}$ the number of edges between vertices $v_i$ and $v_j.$
Since $H$ has no loops, we have $a_{i\,i}=0.$ To define the
circulant foliation $H_n=H_n(G_1,\,G_2,\ldots,G_m)$ we prescribe to
each vertex $v_i$ a circulant graph $G_i=C_n(s_{i,1},\,s_{i,2},\ldots,s_{i,k_i}).$ 
In the case $G_i=C_{n}(\emptyset)$ we set $k_{i}=0$. The
\textit{circulant foliation} $H_n=H_n(G_1,\,G_2,\ldots,G_m)$ over
$H$ with fibers $G_1,\,G_2,\ldots,G_m$ is a graph with the vertex
set $V(H_n)=\{(k,\,v_i)\ | \,k=1,2,\ldots n,\,i=1,2,\ldots,m\},$
where for a fixed $k$ the vertices $(k,\,v_i)$ and $(k,\,v_j)$ are
connected by $a_{i\,j}$ edges, while for a fixed $i,$ the vertices
$(k,\,v_i),\,k=1,2,\ldots n$ form a graph $C_n(s_{i,1},\,s_{i,2},\ldots,s_{i,k_i})$ 
in which the vertex $(k,\,v_i)$ is adjacent to the vertices 
$(k\pm s_{i,1},v_i),(k\pm s_{i,2},\,v_i),\ldots,(k\pm s_{i,k_i},\,v_i)\,(\textrm{mod}\ n).$
Along the text we will refer to $H$ as a \textit{base graph} for the circulant foliation $H_{n}.$

There is a projection $\varphi:H_n\to H$ sending the vertices
$(k,\,v_i),\,k=1,\ldots,n$ and edges between them to the vertex
$v_i$ and for given $k$ each edge between the vertices $(k,\,v_i)$
and $(k,\,v_j),\,i\neq j$ bijectively to an edge between $v_i$ and
$v_j.$ For each vertex $v_i$ of graph $H$ we have
$\varphi^{-1}(v_i)=G_i,\,i=1,2,\ldots,m.$

Consider an action of the cyclic group $\mathbb{Z}_n$ on the graph
$H_n$ defined by the action $(k,\,v_{i})\to(k+t  ,\,v_{i})$ for any $t \in \mathbb{Z}_n$ 
and $i=1, \ldots, m.$ Then the group $\mathbb{Z}_n$ acts freely on the set of vertices and
the set of edges and the factor graph $H_{n}/\mathbb{Z}_{n}$ is an \textit{equipped graph} 
$\widehat{H}$ obtained from the graph $H$ by attaching $k_i$ loops to each $i$-th vertex of $H.$

By making use of the voltage technique \cite{GrossTucker}, one can
construct the graph $H_n$ in the following way. We put an
orientation to all edges of $\widehat{H}$ including loops. Then we
prescribe the voltage $0$ to all edges of subgraph $H$ of
$\widehat{H}$ and the voltage $s_{i,j}\,\mod n$ to the $j$-th loop
attached to $i$-th vertex of $H.$ The respective voltage covering is
the graph $H_n.$

\bigskip
As the first example, we consider the \textit{sandwich graph}
$SW_n=H_n(G_1,\,G_2,\ldots,G_m).$ The graph $SW_n$ has the path
graph on $m$ vertices $v_1,v_2,\ldots,v_m$ as its base graph $H,$
with prescribed circulant graphs $G_i,\,i=1,2,\ldots,m.$ A very
particular case of this construction, known as $I$-graph $I(n,k,l),$
occurs by taking $m=2,\,G_1=C_n(k)$ and $G_2=C_n(l).$ Also, the
generalized Petersen graph \cite{SS09} arises as $GP(n,k)=I(n,k,1).$
The sandwich of two circulant graphs $H_n(G_1,G_2)$ was investigated
in \cite{AbrBaiGruMed}.

As the second example, we consider the \textit{generalized $Y$-graph} 
$Y_n=Y_n(G_1,G_2,G_3),$ where $G_1,G_2,G_3$ are given circulant 
graphs on $n$ vertices. To construct $Y_n,$ we consider a $Y$-shape 
graph $H$ consisting of four vertices $v_1,v_2,v_3,v_4$ and three edges 
$v_1v_4,v_2v_4,v_3v_4.$ Let $G_4=C_n(\emptyset)$ be the empty graph 
on $n$ vertices. Then, by definition, we put $Y_n=H_n(G_1,G_2,G_3,G_4).$ 
In a particular case, if $G_1=C_n(k),G_2=C_n(l),$ and $G_3=C_n(m),$ 
then the graph $Y_n$ coincides with the $Y$-graph  $Y(n;k,l,m)$ defined 
earlier in \cite{BigIYH, HorBou}.

The third example is the \textit{generalized $H$-graph}
$H_n(G_1,G_2,G_3,G_4,G_5,G_6),$ where $G_1,G_2,$ $G_3,G_4$ are given
circulant graphs and $G_5=G_6=C_n(\emptyset)$ are the empty graphs
on $n$ vertices. In this case, we take $H$ to be the graph with
vertices $v_1,v_2,v_3,v_4,v_5,v_6$ and edges
$v_1v_5,v_5v_3,v_2v_6,v_6v_4,v_5v_6.$ In the case
$G_1=C_n(i),G_2=C_n(j),G_3=C_n(k),G_4=C_n(l),$ we get the graph
$H(n;i,j,k,l)$ investigated in the paper \cite{HorBou}. Shortly, we
will write $H_n(G_1,G_2,G_3,G_4)$ ignoring the last two empty graph
entries.

Recall that the adjacency matrix of the non-empty circulant graph
$C_n(s_1,s_2,\ldots,s_k)$ on the vertices $1,2,\ldots,n$ has the
form $\sum\limits_{p=1}^{k}(\T_{n}^{s_p}+\T_{n}^{-s_p}).$ Empty
circulant graph $C_n(\emptyset)$ has zero matrix as its adjacency matrix. Suppose
that the adjacency matrix of  graph $H$ is
$$A(H)=
\left(\begin{array}{ccccc}
0 & a_{1,2} & a_{1,3} & \ldots & a_{1,m} \\
a_{2,1} & 0 & a_{2,3} & \ldots & a_{2,m} \\
  & \vdots &   &   & \vdots \\
a_{m,1} & a_{m,2} & a_{m,3} & \ldots & 0\\
\end{array}\right).$$
Then, the adjacency matrix of the circulant foliation $H_n=H_n(G_1,\,G_2,\ldots,G_m)$ over a graph $H$ with fibers $G_i=C_n(s_{i,1},\,s_{i,2},\ldots,s_{i,k_i}),\,i=1,2,\ldots,n$ is given by
$$A(H_n)=
\left(\begin{array}{ccccc}
\sum\limits_{p=1}^{k_1}(\T_{n}^{s_{1,p}}+\T_{n}^{-s_{1,p}}) & a_{1,2}I_n & a_{1,3}I_n & \ldots & a_{1,m}I_n \\
a_{2,1}I_n & \sum\limits_{p=1}^{k_2}(\T_{n}^{s_{2,p}}+\T_{n}^{-s_{2,p}}) & a_{2,3}I_n & \ldots & a_{2,m}I_n \\
  & \vdots &  &  & \vdots \\
a_{m,1}I_n & a_{m,2}I_n & a_{m,3}I_n & \ldots & \sum\limits_{p=1}^{k_m}(\T_{n}^{s_{m,p}}+\T_{n}^{-s_{m,p}})\\
\end{array}\right).$$  In the particular case, when $G_i=C_n(\emptyset)$ is the  empty graph,  we have $k_i=0.$  In this case, symbol $\sum\limits_{p=1}^{k_i}(\T_{n}^{s_{m,i}}+\T_{n}^{-s_{m,i}})$ mean   the zero $n \times n$  matrix.

Degree  matrix of graph $H_n$  given by
$$D(H_n)=
\left(\begin{array}{ccccc}
(2k_{1}+d_{1})I_n& 0 & 0 & \ldots &0 \\
0 & (2k_{2}+d_{2})I_n & 0 & \ldots & 0 \\
  & \vdots &  &  & \vdots \\
0 & 0 & 0& \ldots & (2k_{m}+d_{m})I_n\\
\end{array}\right).$$

As a result, the Laplace matrix
matrix $L=D(H_n)-A(H_n)$ of  $H_n$ is given by the formula
\begin{equation}\label{Laplacian}
L=
\left(\begin{array}{ccccc}
L_{1}(\T_{n}) & -a_{1,2}I_n & -a_{1,3}I_n & \ldots & -a_{1,m}I_n \\
-a_{2,1}I_n & L_{2}(\T_{n}) & -a_{2,3}I_n & \ldots & -a_{2,m}I_n \\
  & \vdots & &  & \vdots \\
-a_{m,1}I_n & -a_{m,2}I_n & -a_{m,3}I_n & \ldots & L_{m}(\T_{n})\\
\end{array}\right),
\end{equation}
where $L_{i}(z)=2k_{i}+d_{i}-\sum\limits_{j=1}^{k_{i}}(z^{s_{i,j}}+z^{-s_{i,j}})$  if $G_{i}\neq C_{n}(\emptyset)$ and $L_{i}(z)=d_{i}$  if $G_{i}= C_{n}(\emptyset).$  

\section{Counting characteristic polynomial for graph $H_n$}\label{counting}

Let $H$ be a finite connected graph with the set of vertices 
$V(H)=\{v_1,\,v_2,\ldots,v_m\}.$ Consider the circulant foliation
$H_n=H_n(G_1,\,G_2,\ldots,G_m),$ with based  graph $H$  and  fibres  
$G_i=C_n(s_{i,1},\,s_{i,2},\ldots,s_{i,k_i}),\,i=1,2,\ldots,m.$ 

The aim of this section is to find convenient formula for characteristic polynomial of the Laplace matrix $L$ of $H_n$ defined as  $\chi_n(\lambda) = \det(L - \lambda I),$ where $I$ is the identity matrix of order equal to the number $m n$ of vertices of $H_n.$ The most important case for application is when the size of based graph $H$ is fixed, but the size $n$ of each fibre of foliation is going to the infinity. Then the matrix $L - \lambda I$ is quite large and direct calculation of its determinant became a serious problem. To avoid this, we prefer to calculate  the determinant of generalized Laplacian $L(H,\,X)$ for the based graph $H$ and particular set of variables $X=\{x_1,x_2, \ldots,x_m\}.$ Then the finite product of such determinant whose length does not depend on large $n$ gives us a possibility to find $ \chi_n(\lambda).$
 
We fix some complex number $\lambda$ which is supposed to be an argument of $ \chi_n(\lambda).$ Later, we introduce Laurent polynomial $P_\lambda(z)$ and ordinary polynomial $Q_\lambda(w)$ associated with the matrix  $L - \lambda I.$ They will be related by the equation $P_\lambda(z)=Q_\lambda(w)$ for $w=\frac12(z+\frac{1}{z}).$
 
To do this, we consider generalized Laplacian $L(H,\,X)$ of graph $H$ with 
the set of variables $X=(x_1,x_2,\ldots,x_m).$ We specify $X$ by setting
$x_{i}=2k_{i}+d_i-\lambda-\sum\limits_{j=1}^{k_i}(z^{s_{i,j}}+z^{-s_{i,j}}),$
where $d_i$ is the degree of $v_i$ in $H.$ 

We set $P_\lambda(z)=\det(L(H,\,X)).$ By definition of circulant graphs, if 
$G_{i}\neq C_{n}(\emptyset)$ we have $1\leq s_{i,1}<s_{i,2}<\cdots<s_{i,k_{i}}.$ 
In this case, $x_{i}$ is a Laurent polynomial in $z$ with leading term $-z^{s_{i,k_{i}}}.$ 
For $G_{i}=C_{n}(\emptyset)$ we have $x_{i}=d_{i}-\lambda.$ This is a constant polynomial 
in $z.$ We note that $P_\lambda(z)$ is also a Laurent polynomial.

For further considerations, we need to investigate the structure of leading term 
of the polynomial $P_\lambda(z).$ This term comes from the product of polynomials
$x_{i}=2k_{i}+d_{i}-\lambda-\sum\limits_{j=1}^{k_{i}}(z^{s_{i,j}}+z^{-{s_{i,j}}})$
with $s_{i,k_{i}}>0$ and determinant of the matrix, obtained from $L(H,\,X)$ 
removing all rows and columns corresponding to the non-empty circulant fibers. 
One can use the following lemma to calculate the leading term explicitly.

\begin{lemma}\label{leadcoef}
Let $V^{\prime}=(v_1,v_2,\ldots,v_{m^{\prime}})$ be the subset
(possibly empty) of vertices of the graph $H$ with trivial circulant
fibers $C_{n}(\emptyset).$ Let $H^{\prime}$ be the subgraph of graph $H$ induced by $V^{\prime}.$ Then the leading term of the polynomial $P_{\lambda}(z)$ is given by the following formula
$(-1)^{m-m^{\prime}}\eta\,z^{s},$ where $\eta=\eta(\lambda)$ is an integer polynomial in $\lambda$ given by the formula
$$\eta(\lambda)=\det(L(H^{\prime},X^{\prime})),\,s=\sum\limits_{j=1}^{m}s_{j,k_{j}},\,
X^{\prime}=(d_{1}-\lambda,d_{2}-\lambda,\ldots,d_{m^{\prime}}-\lambda),$$ 
and $d_{j}$ is a valency of vertex $v_j$ in graph $H.$
\end{lemma}

\begin{proof}
By definition, $P_\lambda(z)=\det(L(H,X)),$ where $X=(x_{1},x_{2},\ldots,x_{m}),\,x_{j}=2k_{i}+d_{i}-\lambda-\sum\limits_{j=1}^{k_{i}}(z^{s_{i,j}}+z^{-s_{i,j}}).$ For $j=1,2,\ldots,m^{\prime}$ we have $x_{j}=d_{j}-\lambda$ as $k_j=0.$ If $j=m^{\prime}+1,\ldots,m,$ the leading terms of the polynomial $x_j$ is the $-z^{s_{j,k_{j}}}.$ Since all other entries of matrix $L(H,X)$ are integer polynomials in $z$ and $\lambda,$ by making use of basic properties of determinants, we  get the statement of lemma. \end{proof}

\begin{remark}\label{eta1}
If there are no trivial fibers $C_{n}(\emptyset)$ in circulant foliation $H_{n},$ that is $m^{\prime}=0,$ we always have $\eta=1.$ Also, if all fibers in circulant foliation $H_{n}$ are trivial, then $\eta(\lambda)=\det(L(H)-\lambda I)$ coincides with the characteristic polynomial of the base graph $H.$
\end{remark}

Now, we consider one more specification $L(H,\,W)$ for generalized
Laplacian of $H$ with the set  $W=(\omega_1,\omega_2,\ldots,\omega_m),$ where
$\omega_{i}=2k_{i}+d_i-\lambda-\sum\limits_{j=1}^{k_i}2\,T_{s_{i,j}}(w)$ and
$T_k(w)=\cos(k\arccos w)$ is the Chebyshev polynomial of the first
kind. See \cite{MasHand} for the basic properties of the Chebyshev
polynomials. We set $Q_\lambda(w)=\det(L(H,\,W)).$ Then $Q_\lambda(w)$ is a polynomial in $w$ 
of degree $s=s_{1,k_1}+s_{2,k_2}+\cdots+s_{m,k_m}.$ Since
$\frac{1}{2}(z^n+z^{-n})=T_n(\frac{1}{2}(z+z^{-1})),$ we have
$P_\lambda(z)=Q_\lambda(w),$ where $w=\frac{1}{2}(z+z^{-1}).$ By definition of
$Q_\lambda(w)$ we have
\begin{equation}\label{formula1}
Q_\lambda(w)=\det\left(
\begin{array}{ccccc}
\omega_1 & -a_{1,2} & -a_{1,3} & \ldots & -a_{1,m} \\
-a_{2,1} & \omega_2 & -a_{2,3} & \ldots & -a_{2,m} \\
  & \vdots &  &  & \vdots \\
-a_{m,1} & -a_{m,2} & -a_{m,3} & \ldots & \omega_m
\end{array}\right),\end{equation}
where $\omega_{i}=2k_i+d_i-\lambda-\sum\limits_{j=1}^{k_i}2T_{s_{i,j}}(w),\,i=1,2,\ldots,m.$

\begin{remark}\label{rm1} In notation of Lemma~\ref{leadcoef}, the leading term of the polynomial $Q_\lambda(w)$ is $(-1)^{m-m^{\prime}}2^{s}\eta\,w^{s}.$ Equivalently, up to sign,
$\eta=\eta(\lambda)$ is the leading  coefficient of the polynomial  
$Q_\lambda(\frac{w}{2})$ in $w.$ 

It should be noted that  $\widetilde{Q}(\lambda, w)=Q_\lambda(\frac{w}{2})$ is   a polynomial in two variables $\lambda$  and $w$ with integer coefficients.
\end{remark}
\bigskip

The main result of this section is the following theorem.   
\begin{theorem}\label{theorem1}
Characteristic polynomial $\chi_n(\lambda) = \det(L - \lambda I)$ of the Laplace matrix $L=L(H_n)$ of graph
$H_{n}(G_1,\,G_2,\ldots, G_m),$ up to sign, is given by the formula
$$\eta(\lambda)^{n}\prod_{p=1}^{s}(2T_n(w_p)-2),$$ where
$s=s_{1,k_1}+s_{2,k_2}+\cdots+s_{m,k_m},\ w_p,\,p=1,2,\ldots,s$ are all the roots of the equation $Q_\lambda(w)=0$ and $\eta=\eta(\lambda)$ is the leading coefficient of the polynomial  
$Q_\lambda(\frac{w}{2})$ in $w.$
\end{theorem}

\begin{proof}
 By the formula (\ref{Laplacian})  we get
$$L-\lambda I=
\left(\begin{array}{ccccc}
\Lambda_{1}(\T_{n}) & -a_{1,2}I_n & -a_{1,3}I_n & \ldots & -a_{1,m}I_n \\
-a_{2,1}I_n & \Lambda_{2}(\T_{n}) & -a_{2,3}I_n & \ldots & -a_{2,m}I_n \\
  & \vdots & &  & \vdots \\
-a_{m,1}I_n & -a_{m,2}I_n & -a_{m,3}I_n & \ldots & \Lambda_{m}(\T_{n})\\
\end{array}\right),$$
where $\Lambda_{i}(z)=2k_{i}+d_{i}-\lambda-\sum\limits_{j=1}^{k_{i}}(z^{s_{i,j}}+z^{-s_{i,j}}),\,i=1,\ldots,m.$

Our  aim is to find $ \chi_n(\lambda) = \det(L - \lambda I).$  The latter is just the product of eigenvalues of  matrix  $L - \lambda I.$ 

The eigenvalues of circulant matrix $\T_{n}$ are
$\varepsilon_{n}^{j},\,j=0,1,\ldots,n-1,$ where
$\varepsilon_n=e^\frac{2\pi i}{n}.$ Since all of them are distinct,
the matrix $\T_{n}$ is conjugate to the diagonal matrix
$\mathbb{T}_{n}=diag(1,\varepsilon_{n},\ldots,\varepsilon_{n}^{n-1})$
with diagonal entries
$1,\varepsilon_{n},\ldots,\varepsilon_{n}^{n-1}$. To find spectrum
of $L-\lambda I,$ without loss of generality, one can replace $\T_{n}$ with
$\mathbb{T}_{n}.$ Then all $n\times n$ blocks of the above matrix are diagonal
matrices. This essentially simplifies the problem of finding
eigenvalues. Indeed, let $\mu$ be an
eigenvalue of $L- \lambda I$ and let $(u_{1},u_{2},\ldots, u_{m})$ with
$u_{i}=(u_{i,1},u_{i,2}\ldots,u_{i,n})^t,\,i=1,\ldots,m$ be the
respective eigenvector. Then we have the following system of
equations
\begin{equation}\label{equationL1}\left(\begin{array}{ccccc}
\Lambda_{1}(\mathbb{T}_{n}) -\mu\,I_{n}& -a_{1,2}I_{n} & -a_{1,3}I_{n} & \ldots & -a_{1,m}I_{n} \\
-a_{2,1}I_{n} & \Lambda_{2}(\mathbb{T}_{n})-\mu\,I_{n} & -a_{2,3}I_{n} & \ldots & -a_{2,m}I_{n} \\
 & \vdots & &  & \vdots \\
-a_{m,1}I_{n} & -a_{m,2}I_{n} & -a_{m,3}I_{n} & \ldots & \Lambda_{m}(\mathbb{T}_{n})-\mu\,I_{n}\\
\end{array}\right)
\left(\begin{array}{c}
u_{1}\\
u_{2}\\
\vdots \\
u_{m}\\
\end{array}\right)=0.\end{equation}
Recall that all blocks in the matrix under consideration are diagonal
$n\times n$-matrices and the $(j,j)$-th entry of $\mathbb{T}_{n}$ is equal to $\varepsilon_n^{j-1}.$

Hence, the equation (\ref{equationL1}) splits into $n$ equations

\smallskip

\begin{equation}\label{equationL2}\left(\begin{array}{ccccc}
\Lambda_{1}(\varepsilon_n^{j}) -\mu& -a_{1,2}  & -a_{1,3}  & \ldots & -a_{1,m}  \\
-a_{2,1} & \Lambda_{2}(\varepsilon_n^{j})-\mu  & -a_{2,3}  & \ldots & -a_{2,m}  \\
 & \vdots & &  & \vdots \\
-a_{m,1}  & -a_{m,2}  & -a_{m,3}  & \ldots & \Lambda_{m}(\varepsilon_n^{j})-\mu \\
\end{array}\right)
\left(\begin{array}{c}
u_{1,j+1}\\
u_{2,j+1}\\
\vdots \\
u_{m,j+1}\\
\end{array}\right)=0,\end{equation}
$j=0,1,\ldots,n-1$. Each equation gives $m$ eigenvalues of $L-\lambda I,$  say $\mu_{1,j},\mu_{2,j},\ldots,\mu_{m,j}.$ To find these eigenvalues we set

\smallskip

\begin{equation}\label{equationL23}S(z,\mu)=\det \left(\begin{array}{ccccc}
\Lambda_{1}(z) -\mu & -a_{1,2}  & -a_{1,3}  & \ldots & -a_{1,m}  \\
-a_{2,1} & \Lambda_{2}(z)-\mu  & -a_{2,3}  & \ldots & -a_{2,m}  \\
 & \vdots & &  & \vdots \\
-a_{m,1}  & -a_{m,2}  & -a_{m,3}  & \ldots & \Lambda_{m}(z)-\mu \\
\end{array}\right).
\end{equation}

Then $\mu_{1,j},\mu_{2,j},\ldots,\mu_{m,j}$ are roots of the equation
\begin{equation}S(\varepsilon_n^{j},\mu)=0.\end{equation}
In particular, by Vieta's theorem, the product
$\kappa_{j}=\mu_{1,j}\mu_{2,j}\cdots\mu_{m,j}$ is given by
the formula $\kappa_{j}=S(\varepsilon_n^j,0)=P_\lambda(\varepsilon_n^j),$  where
$P_\lambda(z)$ is defined at the beginning of this section.

Now, for any $j=0,\ldots, n-1,$ matrix $L-\lambda I$ has $m$ eigenvalues
$\mu_{1,j},\mu_{2,j},\ldots,\mu_{m,j}$ satisfying the order $m$ algebraic 
equation $S(\varepsilon_n^{j},\mu)=0.$ One can see that $P_\lambda(z)=S(z,0)$ 
is the characteristic polynomial for generalized Laplacian $L(H,X),$ where 
$X=(x_{1},\ldots, x_{m}),$ and  $x_{i}=2k_{i}+d_i-\lambda-\sum\limits_{j=1}^{k_i}(z^{s_{i,j}}+z^{-s_{i,j}}).$

In particular, $P_\lambda(\varepsilon_n^j)$ is the characteristic polynomial for generalized 
Laplacian $L(H,X_j),$  where $X_j=(x_{1j},\ldots, x_{m j})$ with  
$$x_{i j}=2k_{i}+d_i-\lambda-\sum\limits_{j=1}^{k_i}(\varepsilon_n^{j\, s_{i,j}}+\varepsilon_n^{-j\,s_{i,j}}).$$  
For $j=0$ we have $X_0=(d_1-\lambda,\ldots, d_m-\lambda)$ and 
$$P_\lambda(1)=\det L(H,X_0)=\chi_H(\lambda).$$
Now we get
\begin{equation}\label{tauH}
\chi_n(\lambda)=\prod\limits_{j=0}^{n-1}\mu_{1,j}\mu_{2,j}\cdots\mu_{m,j}=
\prod\limits_{j=0}^{n-1}\kappa_j=\prod\limits_{j=0}^{n-1}P_\lambda(\varepsilon_n^j).
\end{equation}

Note that the product of eigenvalues
$\kappa_0=P_\lambda(1)=\mu_{1,0}\mu_{2,0}\cdots\mu_{m,0}$ is equal to $\chi_H(\lambda ),$ 
where $\chi_H(\lambda )$ is the Laplace characteristic polynomial of graph $H.$ Hence
\begin{equation}\label{multiple}
\chi_n(\lambda )=\chi_H(\lambda )\prod\limits_{j=1}^{n-1}P_\lambda(\varepsilon_{n}^{j}).
\end{equation}
By definition, $\eta(\lambda)$ is the leading coefficient of two variable Laurent polynomial 
$P_\lambda(z)$ in $z.$ So, $\eta(\lambda)\neq 0.$ 
To continue the proof, we replace the Laurent polynomial $P_\lambda(z)$ by
$\widetilde{P}_\lambda(z)=(-1)^{m-m^{\prime}}\frac{z^{s}}{\eta(\lambda)}P_\lambda(z),$ where
$m,\,m^{\prime}$ and $\eta(\lambda)$ are given by Lemma~\ref{leadcoef}. Then
$\widetilde{P}_\lambda(z)$ is a monic polynomial in $z$ of the degree $2s$ with the
same roots as $P_\lambda(z).$ We note that
\begin{equation}\label{newP}
\eta(\lambda)^{n}\prod\limits_{j=0}^{n-1}\widetilde{P_\lambda}(\varepsilon_{n}^{j})=(-1)^{(m-m^{\prime})n}\varepsilon_{n}^{\frac{(n-1)n}{2}s} \prod\limits_{j=0}^{n-1}P_\lambda(\varepsilon_{n}^{j})=(-1)^{(m-m^{\prime})n+s(n-1)}\prod\limits_{j=0}^{n-1}P_\lambda(\varepsilon_{n}^{j}).
\end{equation}

All roots of polynomials $\widetilde{P_\lambda}(z)$ and $Q_\lambda(w)$ are
$z_{1},1/z_{1},z_{2},1/z_{2},\ldots,z_{s},1/z_{s}\textrm{ and
}w_{j}=\frac{1}{2}(z_{j}+z_{j}^{-1}),\,j=1,\ldots,s,$ respectively.
Also, we can recognize the complex numbers $\varepsilon_{n}^{j},\,j=0,\ldots,n-1$ 
as the roots of polynomial $z^n-1.$ By the basic properties of resultant (\cite{PrasPol},
Ch.~1.3) we have
\begin{eqnarray}\label{Hlemma}
\nonumber &&\prod\limits_{j=0}^{n-1}\widetilde{P_\lambda}(\varepsilon_{n}^{j})=\textrm{Res}(\widetilde{P_\lambda}(z),z^{n}-1) =\textrm{Res}(z^{n}-1,\widetilde{P_\lambda}(z))\\
&&=\prod\limits_{z:\widetilde{P_\lambda}(z)=0}(z^{n}-1)=\prod\limits_{j=1}^{s}(z_{j}^{n}-1)(z_{j}^{-n}-1)
\\
\nonumber &&=\prod\limits_{j=1}^{s}(2-z_{j}^{n}-z_{j}^{-n})=(-1)^{s}\prod\limits_{j=1}^{s}(2T_{n}(w_{j})-2).
\end{eqnarray}
Combine (\ref{tauH}), (\ref{newP}) and (\ref{Hlemma}) we have the following formula 
 \begin{equation}\label{before}
\chi_n(\lambda)=(-1)^{(s+m-m^{\prime})n}\eta(\lambda)^{n}\prod\limits_{j=1}^{s}(2T_{n}(w_{j})-2).
\end{equation}

We finish the proof of the theorem.
\end{proof}
The following corollary is well known. See papers \cite{FengKwakLee, MizunoSato} and others.
\bigskip

\begin{corollary} Characteristic polynomial $ \chi_n(\lambda)$ of graph
$H_{n}(G_1,\,G_2,\ldots, G_m)$ is a multiple of characteristic polynomial $\chi _H(\lambda)$ of graph $H.$
\end{corollary}

\begin{proof}
Indeed, by formula~(\ref{multiple}) we have to show that the product
$R=\prod\limits_{j=1}^{n-1}P_\lambda(\varepsilon_n^j)$ is a  polynomial in $\lambda.$ This
follows from the fact that $R=\textrm{Res}(P_{\lambda}(z),\frac{z^{n}-1}{z-1})$ and Vieta's theorem.    
\end{proof}

\section{Finding a square in a characteristic polynomial}\label{arithmetic}

Let $H$ be a finite connected graph on $m$ vertices. Consider the
circulant foliation $H_n=H_n(G_1,\,G_2,\ldots,G_m),$ where
$G_i=C_n(s_{i,1},\,s_{i,2},\ldots,s_{i,k_i}),\,i=1,2,\ldots,m.$ Denoted by $Q_\lambda(w)$ the polynomial given in the formula (\ref{formula1}).
\bigskip

\begin{theorem}\label{lorenzini}  
Let $\chi_n(\lambda) $ and $ \chi _{H}(\lambda)$ be the Laplace characteristic polynomials of graphs
$H_n$ and  $H$  respectively. Then there exists a sequence of integer polynomials $a_n(\lambda)$ such that
\begin{enumerate}
\item[$1^{\circ}$] $\chi_n(\lambda) = \chi _{H}(\lambda)\,a_n(\lambda)^{2}$ if $n$ is odd,
\item[$2^{\circ}$] $\chi_n(\lambda) =Q_\lambda(-1)\chi _{H}(\lambda)\,a_n(\lambda)^{2}$ if $n$ is even.
 \end{enumerate}

\end{theorem}

\begin{proof} By
formula~(\ref{multiple}) we have
$\chi_n(\lambda)=\chi _{H}(\lambda)\prod_{j=1}^{n-1}P_\lambda(\varepsilon_{n}^{j}).$ Note that
$P_\lambda(\varepsilon_{n}^{j})=
P_\lambda(\varepsilon_{n}^{n-j})$ for $j=1,\ldots, n-1.$

Define $c_{n}(\lambda)=\prod\limits_{j=1}^{\frac{n-1}{2}}P_\lambda(\varepsilon_{n}^{j})$
if $n$ is odd and $d_{n}(\lambda)=\prod\limits_{j=1}^{\frac{n}{2}-1}P_\lambda(\varepsilon_{n}^{j})$
if $n$ is even. Since $P_\lambda(\varepsilon_n^j)$ is the characteristic polynomial for generalized Laplacian $L(H,X_j),$ each algebraic number $\lambda$ is contained in the above products together with all of its Galois conjugate elements \cite{Lor}. Therefore, the products $c_{n}(\lambda)$ and $d_{n}(\lambda)$  are integer polynomials  in $\lambda$.

Moreover, if $n$ is even we get $P_\lambda(\varepsilon_{n}^{\frac{n}{2}})=P_\lambda(-1)=Q_\lambda(-1).$
We note that $Q_\lambda(-1)$ is always an integer polynomial in $\lambda.$ The exact formula for it is given below in Remark~\ref{rm2}.

Now, we have $\chi_n(\lambda) =\chi _{H}(\lambda)\,c_{\lambda}(n)^{2}$  if $n$ is odd, and $\chi_n(\lambda) =Q_\lambda(-1)\chi _{H}(\lambda)\,d_{\lambda}(n)^{2}$  if $n$ is even. Putting $a(n)=c_{\lambda}(n)$ in the first case, and $a(n)=d_{\lambda}(n)$  in the
second we finish the proof of the theorem.
\end{proof}

\begin{remark}\label{rm2} We note that  $Q_\lambda(-1)$ is  an integer  polynomial in $\lambda.$
Indeed, $Q_\lambda(w)=\det\,L(H,W),$ where $W=(w_{1},w_{2},\ldots,w_{m})$
and $w_{i}=2k_{i}+d_{i}-\lambda-\sum\limits_{j=1}^{k_{i}}2T_{s_{i,j}}(w).$
Denoted by $t_i$ the number of odd elements in the sequence
$s_{i,1},\,s_{i,2},\ldots,s_{i,k_i}.$ Since $T_{s_{i,j}}(-1)=\cos(s_{i,j}\arccos(-1))=\cos(s_{i,j}\pi)=(-1)^{s_{i,j}},$ for $w=-1$ we have
$$w_i=d_{i}-\lambda+4\sum\limits_{j=1}^{k_{i}}\frac{1-(-1)^{s_{i,j}}}{2}=d_{i}+4t_{i}-\lambda.$$
Hence, $Q_\lambda(-1)=\det\,L(H,W),\text{ where }W=(d_{1}+4t_{1}-\lambda,d_{2}+4t_{2}-\lambda,\ldots,d_{m}+4t_{m}-\lambda).$
\end{remark}

\section{Application}\label{application}

\subsection{Geometrical meaning of coefficients of $\chi_n(\lambda).$}\label{coefficients}

A {\it tree} is a connected undirected graph without cycles. A {\it spanning tree} 
in a graph $G$ is a subgraph that is a tree containing all the vertices of $G.$ 
A {\it rooted tree} is a tree with one marked vertex called root. A {\it rooted forest} 
is a graph whose connected components are rooted trees. A {\it rooted spanning
forest} $F$ in a graph $G$ is a subgraph that is a rooted forest containing all the 
vertices of $G.$ 

Let $G$ be a graph on $n$ vertices. Let $\chi_{G}(\lambda)=\det(L(G)-\lambda\,I_{n})$ be the characteristic polynomial of matrix $L(G),$ which is Laplace matrix of graph $G.$ Its extended form is
$$\chi_{G}(\lambda)=(-1)^n \lambda^n+c_{n-1}\lambda^{n-1}+\ldots+c_1\lambda.$$ 
By celebrated Kirchhoff theorem, the number $\tau(G)$ of  spanning trees in connected  $G$ is the product of non-zero roots of $\chi_{G}(\lambda)$  divided by the number of vertices of $G.$  
That is $$\tau(G)=-\frac{\chi_{G}^{\prime}(0)}{n}=-\frac{c_1}{n}.$$

We note that since all eigenvalues of $L(G)$ are non-negative, the sequence of coefficients 
$c_1, \ldots, c_{n-1}, c_n=(-1)^n$ so alternating. The theorem by Kelmans and Chelnokov 
\cite{KelChel} states that the absolute value of coefficient $c_{k}$ of $\chi_{G}(\lambda)$ 
coincides with the number of rooted spanning $k$-forests in the graph $G.$ 

So, the number of rooted spanning forests of the graph $G$ can be found by the formula
\begin{eqnarray}\label{KelmCheln}
f(G)&=&f_{1}+f_{2}+\ldots+f_{n}=|c_{1}-c_{2}+c_{3}-\ldots+(-1)^{n-1}|\\
\nonumber &=&\det(I_{n}+L(G))=\chi_{G}(-1).
\end{eqnarray}

This result was independently obtained by many authors: see, for example, \cite{ChebSham}, \cite{Knill}, and \cite{GrunMed}.  


Let us call a graph $k$-colored if its edges can have $k$ colors. The following two theorems were proved by  Oliver Knill in \cite{Knill}.

\begin{theorem}\label{forestcol} (Forest Coloring Theorem). For a finite simple graph $G$ with Laplacian $L,$ the integer $det(I + k L)$ is the number of rooted $k$-colored spanning forests contained in $G.$
\end{theorem}

One can also look at $k = -1,$ in which case we count forests with odd number of trees with a negative sign. Lets call a graph “even” if it has an even number of edges, and “odd” if it has an odd number of edges.

\begin{theorem}\label{superforestcol}(Super Forest Coloring Theorem). For a finite simple graph $G$ with Laplacian $L,$ the integer $det(I - k L)$ is the number of $k$-colored rooted even spanning forests minus the number of $k$-colored rooted odd spanning forests in $G.$
\end{theorem}

\subsection{Enumeration of the spanning trees}\label{trees}
Let $\tau(n)$ be the number of spanning trees in the circulant foliation $H_n=H_n(G_1,\,G_2,\ldots,G_m).$ In this point we suppose that $H_n$  is connected.

Since the graph $H_n$ has $nm$ vertices, by Kirchhoff theorem, we have 
$\tau(n)=-\frac{\chi_n^{\prime}(0)}{n m}.$ By Theorem \ref{theorem1}, we get
$\chi_n(\lambda)=\eta(\lambda)^{n}\prod_{p=1}^{s}(2T_n(w_p)-2),$ where
$s=s_{1,k_1}+s_{2,k_2}+\cdots+s_{m,k_m}$ and $w_p=w_p(\lambda),\,p=1,2,\ldots,s$ are
all the roots of the equation $Q_\lambda(w)=0.$ Let $w=w(\lambda)$ be a root of equation $Q_\lambda(w)=0.$ If $\lambda=0,$ the properties of polynomial $Q_0(w)$ are well known. They are represented in the papers \cite{KwonMedMed18, KwonMedMed20, MedMedUMN2023}. In particular, $Q_0(1)=0$ and $Q_0^{\prime}(1)<0.$ So, $Q_0(w)$ has a simple root, say $w_1(0)=1$ and the others roots $w_2(0),\ldots, w_m(0)$ are differ from $1.$  

We set $J(w)=2T_n(w)-2.$ Note that $J(w_1(0))=J(1)=0.$ Also, $\eta(0)\neq 0.$ This inequality follows from Theorem~\ref{theorem1} and the fact that characteristic polynomial $\chi_{n}(\lambda)$ of connected graph $H_n$ has a simple root at $\lambda=0.$ Then, in a small neighborhood of $\lambda=0,$ we have $\frac{1}{\eta(\lambda)^{n}}\chi_{n}(\lambda)=J(w_1)J(w_2)\cdots J(w_s).$ Therefore
$$\frac{1}{\eta(\lambda)^{n}}\chi_{n}^{\prime}(\lambda)-\frac{\eta(\lambda)^{\prime}}{\eta(\lambda)^{n+1}}\chi_{n}(\lambda)=
J^{\prime}(w_1)\,w_1^{\prime}(\lambda)\,J(w_2)\,\ldots\,J(w_s)+J(w_1)\,J^{\prime}(w_2)\,w^{\prime}_{2}(\lambda)\,\ldots\,J(w_s)+\ldots$${ }$$+J(w_1)\,J(w_2)\,\ldots\,J^{\prime}(w_s)\,w^{\prime}_{s}(\lambda).$$ 
Since $w_1(0)=1,\,J(1)=0$ and $\chi_{n}(0)=0,$ we have 

\begin{equation}\label{eta}\frac{1}{\eta(0)^{n}}\chi_{n}^{\prime}(0)=J^{\prime}(1)\cdot w_1^{\prime}(0) \cdot J(w_2)\cdot\ldots\cdot J(w_s),\end{equation}

For some technical reasons, we will use notation $Q(\lambda,w(\lambda))$ instead of previous $Q_\lambda(w).$ Now our aim is to find $w^{\prime}(\lambda),$ where $w(\lambda)$ is a root of the equation  $Q(\lambda,w(\lambda))=0.$

From formula (\ref{formula1}) we obtain   the 
$$(Q(\lambda,w(\lambda))^{\prime}_{\lambda}=\det\left(
\begin{array}{ccccc}
({\omega_{1}})^{\prime}_{\lambda}& -a_{1,2} & -a_{1,3} & \ldots & -a_{1,m} \\
0 & \omega_{2} & -a_{2,3} & \ldots & -a_{2,m} \\
  & \vdots &  &  & \vdots \\
0 & -a_{m,2} & -a_{m,3} & \ldots & \omega_{m}
\end{array}\right)$${}$$+\det\left(
\begin{array}{ccccc}
\omega_{1} & 0 & -a_{1,3} & \ldots & -a_{1,m} \\
-a_{2,1} & ({\omega_{2}})^{\prime}_{\lambda} & -a_{2,3} & \ldots & -a_{2,m} \\
  & \vdots &  &  & \vdots \\
-a_{m,1} & 0& -a_{m,3} & \ldots &\omega_{m}
\end{array}\right)+\ldots + \det\left(
\begin{array}{ccccc}
\omega_{1} & -a_{1,2} & -a_{1,3} & \ldots & 0 \\
-a_{2,1} & \omega_{2} & -a_{2,3} & \ldots & 0\\
  & \vdots &  &  & \vdots \\
-a_{m,1} & -a_{m,2} & -a_{m,3} & \ldots & ({\omega_{m}})^{\prime}_{\lambda}
\end{array}\right)$${}$$=({\omega_{1}})^{\prime}_{\lambda}Q_{1,1}(w) +({\omega_{2}})^{\prime}_{\lambda}Q_{2,2}(w) + . . . + ({\omega_{m}})^{\prime}_{\lambda}Q_{m,m}(w),$$
where $w=w(\lambda),\,\omega_{i}=2k_i+d_i-\lambda-\sum\limits_{j=1}^{k_i}2T_{s_{i,j}}(w),\,i=1,2,\ldots,m$ and $Q_{i,i}(w)$ is the $(i, i)$-the minor of the matrix in formula (\ref{formula1}).  By  making use of the identity $T^{\prime}_{s}(w)=s U_{s-1}(w),$ we have
$$(\omega_{i})^{\prime}_\lambda=-(1+\sum\limits_{j=1}^{k_i}2 s_{i,j} U_{s_{i,j}-1}(w) w^{\prime}(\lambda)).$$  Hence, since $w=w(\lambda)$  is a root of the equation  $Q(\lambda,w(\lambda))=0,$ we obtain     

\begin{equation}\label{Q}Q(\lambda,w(\lambda))^{\prime}_{\lambda}=-\sum_{i=1}^m(1+\sum\limits_{j=1}^{k_i}2s_{i,j} U_{s_{i,j}-1}(w) w^{\prime}(\lambda)) Q_{i,i}(w)=0.
\end{equation}

We restrict ourself by very particular root $w(\lambda)=w_1(\lambda)$ of equation   $Q(\lambda,w(\lambda))=0,$ satisfying the condition $w_1(0)=1.$ We put $\lambda=0.$ Then $w=w_1(0)=1$ and the matrix in formula (\ref{formula1}) coincides with the Laplacian of graph $H. $ By the Kirchhoff Matrix-Tree theorem we have
$$Q_{1,1}(1) = Q_{2,2}(1) = \ldots = Q_{m,m}(1) = \tau(H),$$ where $\tau(H)$ is the number of spanning trees in the graph $H.$ Moreover, by basic properties of Chebyshev polynomials, $U_{s_{i,j}-1}(1)=s_{i,j}.$  Then, by equation (\ref{Q}) for $\lambda=0$ we have
\begin{equation}\label{QW}(m+\sum_{i=1}^m\sum\limits_{j=1}^{k_i}2s_{i,j}^2 w^{\prime}(0))\tau(H)=0.
\end{equation}
Since $H_{n}$ is connected, the base graph $H$ is also connected. Hence, the number of spanning trees $\tau(H)>0.$ As a result, we get
\begin{equation}\label{wprime}w^{\prime}(0)=w^{\prime}_1(0)=-\frac{m}{
\sum\limits_{i=1}^m\sum\limits_{j=1}^{k_i}2s_{i,j}^2}.
\end{equation}

Since $J(w)=2T_n(w)-2,$ we have $J^{\prime}(w)=2n U_{n-1}(w)$ and $J^{\prime}(1)=
2n^2.$ Then, by formulas (\ref{eta}) and (\ref{wprime}) we obtain
$$\tau(n)=-\frac{\chi^{\prime}_n(0)}{n\,m}=-\frac{\eta(0)^{n}}{n\,m}\cdot J^{\prime}(1)\cdot w_1^{\prime}(0) \cdot J(w_2)\cdot\ldots\cdot J(w_s)$$
$$=-\frac{\eta(0)^{n}}{n\,m} \cdot 2n^2\cdot (-\frac{m}{
\sum\limits_{i=1}^m\sum\limits_{j=1}^{k_i}2s_{i,j}^2})\cdot J(w_2)\cdot\ldots\cdot J(w_s)$${}
$$=\frac{n\cdot \eta(0)^{n}}{q} \cdot J(w_2)\cdot\ldots\cdot J(w_s),$$
where $q=\sum\limits_{i=1}^m\sum\limits_{j=1}^{k_i}s_{i,j}^2$ and $w_2=w_2(0),\ldots, w_s=w_s(0)$ are  all the roots of the equation $Q_0(w)=0$ different from $1.$

As a result, we have the following theorem obtained earlier in \cite{KwonMedMed18} and \cite{MedMedUMN2023}.   

\begin{theorem}\label{theoremtrees}
The number of spanning trees $ \tau(H_n)$ in circulant foliation  
$$H_n=H_{n}(G_1,\,G_2,\ldots, G_m)$$ is given by the formula
$$\tau(H_n)=\frac{n \,|\eta(0)|^{n}}{q}\prod_{p=2}^{s}|2T_n(w_p)-2|,$$ where
$s=s_{1,k_1}+s_{2,k_2}+\cdots+s_{m,k_m},\,q=\,\sum\limits_{i=1}^m\sum\limits_{j=1}^{k_i}s_{i,j}^2$, $\eta(0)$ is the leading coefficient of the polynomial  $Q_0(\frac{w}{2}),$ and $\ w_p,\,p=2,3,\ldots,s$ are
all the roots of the equation $Q_0(w)=0$ different from $1.$
\end{theorem}
\subsection{Enumeration of the rooted spanning forests}\label{forests}

One more consequence of main Theorem \ref{theorem1} is the following results established 
in \cite{GrunKwonMed}. Indeed, since the number of rooted spanning forests $F(H_n)$ 
in circulant foliation $H_n$ is given by $\chi_{n}(-1),$ we obtain the following result.

\begin{theorem}\label{theoremforests}
The number of rooted spanning forests $F(H_n)$ in circulant foliation  
$$H_n=H_{n}(G_1,\,G_2,\ldots, G_m)$$ is given by the formula
$$F(H_n)=|\eta(-1)|^{n}\prod_{p=1}^{s}|2T_n(w_p)-2|,$$ where
$s=s_{1,k_1}+s_{2,k_2}+\cdots+s_{m,k_m},\,\eta(-1)$ is the leading coefficient of the polynomial 
$Q_{-1}(\frac{w}{2}),$ and $\ w_p,\,p=1,2,\ldots,s$ are
all the roots of the equation $Q_{-1}(w)=0.$  
\end{theorem}

\subsection{Enumeration  of the rooted $k$-colored spanning forests}\label{forests} 
By Theorem \ref{forestcol}, the number of rooted $k$-colored spanning forests contained in 
graph $G$ on $n$ vertices is given by $k^{n}\chi_{G}(-1/k).$ Hence, we have the following result. 

\begin{theorem}\label{theoremkforests}
The number of rooted  $k$-colored spanning forests $F_k(H_n)$ in circulant foliation  
$$H_n=H_{n}(G_1,\,G_2,\ldots, G_m)$$  is given by the formula
$$F_k(H_n)=k^{n\,m}|\eta(-1/k)|^{n}\prod_{p=1}^{s}|2T_n(w_p)-2|,$$ where
$s=s_{1,k_1}+s_{2,k_2}+\cdots+s_{m,k_m},\,\eta(-1/k)$ is the leading coefficient of the polynomial 
$Q_{-1/k}(\frac{w}{2}),$  and  $\ w_p,\,p=1,2,\ldots,s$ are
all the roots of the equation $Q_{-1/k}(w)=0.$  
\end{theorem}

In a similar way, one can find the number $ F_{-k}(H_n)=k^{n\,m}\chi_{H_n}(1/k),$ the geometric meaning of which was explained in Theorem~\ref{superforestcol}.

\section{Examples}\label{examples}

\subsection{Circulant graph $C_{n}(s_{1},s_{2},\ldots,s_{k})$}\label{example1}

We consider the classical circulant graph $C_{n}(s_{1},s_{2},\ldots,s_{k})$ 
as a foliation $H_{n}(G_{1})$ on the one vertex graph $H=\{v_{1}\}$ with the fiber
$G_{1}=C_{n}(s_{1},s_{2},\ldots,s_{k}).$ In this case,
$d_{1}=0,\,L(H,X)=(x_{1}),\,P(z)=2k-\lambda-\sum\limits_{p=1}^{k}(z^{s_p}+z^{-s_p})$
and its Chebyshev transform is $Q(w)=2k-\lambda-\sum\limits_{p=1}^{k}2T_{s_p}(w).$ 
Different aspects of complexity for circulant graphs were investigated in the papers
\cite{XiebinLinZhang, Golin2010, MedMed17}. 
The number of rooted spanning forests for circulant graphs is investigated in \cite{GrunMed}.

{\bf Cyclic graph $C_{n}=C_{n}(1).$}
The equation  $Q_\lambda(w)=2-\lambda-2T_1(w)=0$ has  only one root $w=\frac{2-\lambda}{2}$ and $\eta=1.$  Hence,  $\chi_{n}(\lambda)=2T_{n}(\frac{2-\lambda}{2})-2.$ 

{\bf Circulant graph $C_{n}(1,2).$}
We have to solve  the equation  $Q_\lambda(w)=4-\lambda-2T_{1}(w)-2T_{2}(w)=6-\lambda- 2w-4 w^2=0.$ It has two solutions $w_{1}=\frac{1}{4}\left(-1+\sqrt{25-4\lambda}\right)$ and $w_{2}=\frac{1}{4}\left(-1-\sqrt{25-4\lambda}\right).$ By  Remark \ref{rm1}, $\eta=1.$ As a result, we obtain $\chi_{n}(\lambda)=\left(2T_n(w_{1})-2\right) \left(2T_n(w_{2})-2\right).$

\subsection{$I$-graph $I(n,k,l)$ and the generalized Petersen graph $GP(n,k)$}

Let $H$ be a path graph on two vertices, $G_{1}=C_{n}(k)$ and
$G_{2}=C_{n}(l).$ Then $I(n,k,l)=H_{n}(G_{1},G_{2})$ and
$GP(n,k)=I(n,k,1).$ We get $P(z)=(3-\lambda-z^{k}-z^{-k})(3-\lambda-z^{l}-z^{-l})-1$
and $Q(w)=(3-\lambda-2T_{k}(w))(3-\lambda-2T_{l}(w))-1.$ The arithmetical and
asymptotical properties of complexity for $I$-graphs were studied in
\cite{Ilya, AbrBaiGruMed, KwonMedMed17}.

{\bf Graph $I(n,1,1).$} Equation $Q_\lambda(w)=(3-\lambda-2w)^2-1=0$ has two solutions
$w_{1}=\frac{2-\lambda}{2}$ and $w_{2}=\frac{4-\lambda}{2}.$ Here, $\eta=1.$ Hence,
$$\chi_{n}(\lambda)=\left(2T_n(\frac{2-\lambda}{2})-2\right) \left(2T_n(\frac{4-\lambda}{2})-2\right).$$

\subsection{Sandwich of $m$ circulant graphs}

Consider a path graph $H$ on $m$ vertices.
Then $H_{n}(G_{1},G_{2},\ldots,G_{m})$ is a {\it sandwich graph} of circulant graphs
$G_{1},G_{2},\ldots,G_{m}.$ Here $d_{1}=d_{m}=1$ and $d_{i}=2,\,i=2,\ldots,m-1.$  In this  case,  base graph $H$ is the path
graph on $m$ vertices $v_1,v_2,\ldots,v_m.$ 
Let
$L(H,\,X)$ be the generalized Laplacian of graph $H$ with the set of
variables $X=(x_1,x_2,\ldots,x_m).$ We set $D(x_1,x_2,\ldots,x_m)=\det(L(H,\,X)).$
Then $$D(x_1,x_2,\ldots,x_m)=
\det\left(\begin{array}{ccccccc}
x_1 & -1  & 0  & \ldots & 0 & 0& 0 \\
-1  & x_2 & -1 & \ldots & 0 & 0& 0 \\
& \vdots & & & & \vdots & \\
0 & 0 & 0 & \ldots & -1 & x_{m-1} & -1\\
0 & 0 & 0 & \ldots &  0 & -1 & x_m\end{array}\right).$$ By direct calculation we obtain
$$D(x_1,x_2,\ldots,x_m)=x_1D(x_2,\ldots,x_m)-D(x_3,\ldots,x_m),\, D(x_1)=x_1,\,D(x_1,x_2)=x_1x_2-1.$$
Recall that $Q_\lambda(w)=D(w_1,w_2,\ldots,w_m),$ where $w_{i}=2k_{i}+d_i-\lambda-\sum\limits_{j=1}^{k_i}2\,T_{s_{i,j}}(w).$ Then $Q_\lambda(-1)=D(d_1+4t_1-\lambda,d_2+4t_2-\lambda,\ldots, d_m+4t_m-\lambda),$
where $w_i$ and $t_i$ are the same as in Remark~\ref{rm2}. The particular case $H_{n}(G_{1},G_{2})$ is studied in recent paper \cite{AbrBaiGruMed}.

{\bf Graph $H_n(C_{n}(1,2),C_{n}(1,2)).$}  We have $Q_\lambda(w)=(3-\lambda-2T_1(w)-2T_2(w))^2-1.$ Polynomial $Q_\lambda(w)$ has four roots
$w_{1,2}=\frac{1}{4}\left(-1\pm\sqrt{25-4\lambda}\right)$ and $w_{3,4}=\frac{1}{4}\left(-1\pm\sqrt{33-4\lambda}\right).$ By Remark \ref{rm1}, $\eta=1.$ Hence
$$\chi_{n}(\lambda)=\left(2T_n(w_{1})-2\right) \left(2T_n(w_{2})-2\right)\left(2T_n(w_{3})-2\right) \left(2T_n(w_{4})-2\right).$$
One can notice an interesting fact that the $\chi_{n}(6)$ is periodic sequence with period 6. 

{\bf Graph $H_{n}(C_{n}(1),C_{n}(1),C_{n}(1)).$} Here $$Q_\lambda(w)=-(-5 + 2 w + \lambda) (-3 + 2 w + \lambda) (-2 + 2 w + \lambda)\text{ and }\eta=1.$$ 
The equation $Q_\lambda(w)=0$ has three solutions. Hence,
$$\chi_{n}(\lambda)=\left(2T_n(\frac{2-\lambda}{2})-2\right) \left(2T_n(\frac{3-\lambda}{2})-2\right)\left(2T_n(\frac{5-\lambda}{2})-2\right) .$$

\subsection{Generalized $Y$-graph}

Consider the generalized $Y$-graph $Y_{n}(G_{1},G_{2},G_{3}),$ where $G_{i}=C_{n}(s_{i,1},\,s_{i,2},\ldots,s_{i,k_i}),\,i=1,2,3.$ Here
$$Q_\lambda(w)=4A_{1}(w)A_{2}(w)A_{3}(w)-A_{1}(w)A_{2}(w)-A_{1}(w)A_{3}(w)-A_{2}(w)A_{3}(w),$$
where $A_{i}(w)=2k_{i}+d_i-\lambda-\sum\limits_{j=1}^{k_{i}}2T_{s_{i,j}}(w).$

In the particular case $G_{1}=G_{2}=G_{3}=C_{n}(1)$ we have
$Y$-graph $Y(n;1,1,1).$  Then $Q_\lambda(w)=(-3 + 2 w + \lambda)^2 (6 -6 w -6 \lambda+2w \lambda+\lambda^2),\,\eta=3-\lambda.$
The  Laplace characteristic polynomial is given by the formula
$$\chi_{n}(\lambda)=(3-\lambda)^n\left(2T_n\left(\frac{6-6\lambda+\lambda^2}{2\,(3-\lambda)}\right)-2\right) \left(2T_n\left(\frac{3-\lambda}{2}\right)-2\right)^2.$$

Also, by Theorem~\ref{lorenzini} there exists an integer sequence $a(n)$ such that the number of rooted spanning forests $f(n)=\chi_{n}(-1)$ satisfies relations $f(n)=f(1)a(n)^2$ if $n$ is odd and $f(n)=21f(1)a(n)^2$ if $n$ is even. Here the multiple $21$ is a square-free part of $Q_{-1}(-1)$ and
$f(1)=20$ is the number of rooted spanning forests in the base $Y$-graph. 
Compare with the results of the paper \cite{GrunKwonMed}.

\subsection{Generalized $H$-graph}

Consider the generalized $H$-graph $H_{n}(G_{1},G_{2},G_{3},G_{4}),$ where $G_{i}=C_{n}(s_{i,1},\,s_{i,2},\ldots,s_{i,k_i}),\,i=1,2,3,4.$ Now we have
$$Q_\lambda(w)=A_{1}(w)A_{2}(w)A_{3}(w)A_{4}(w)\left((4-\frac{1}{A_{1}(w)}-\frac{1}{A_{2}(w)})
(4-\frac{1}{A_{3}(w)}-\frac{1}{A_{4}(w)})-1\right)$$ where $A_{i}(w)$ are the same as the above.

In case $G_{1}=G_{2}=G_{3}=G_{4}=C_{n}(1)$ we have $H$-graph
$H(n;1,1;1,1).$ Then $$Q_\lambda(w)=(-3 + 2 w + \lambda)^2 (10 - 8 w - 7 \lambda + 
   2 w \lambda + \lambda^2) (4 - 4 w - 5 \lambda + 
   2 w \lambda + \lambda^2)$$ and
$\eta(\lambda)=(-4 + \lambda) (-2 + \lambda).$ The Laplace characteristic polynomial of graph $H(n;
1,1; 1,1)$ is given by 
$$\chi_{n}(\lambda)=\eta(\lambda)^n\left(2T_n\left(\frac{-4+5\lambda-\lambda^2}{2(-2+\lambda)}\right)-2\right) \left(2T_n\left(\frac{-10+7\lambda-\lambda^2}{2(-4+\lambda)}\right)-2\right) \left(2T_n\left(\frac{3-\lambda}{2}\right)-2\right)^2.$$

Also, for the number of rooted spanning forest $f(n)$ and some integer sequence $a(n),$ we have $f(n)=f(1)a(n)^{2}$ for odd $n$ and $f(n)=7f(1)a(n)^2$ for even $n.$ Here $7$ is square-free
part of $Q_{-1}(-1)$ and $f(1)=128$ is the number of rooted spanning forests in the base $H$-graph. 

\subsection{Discrete torus $T_{n,m}=C_n\times C_m$}

We have $T_{n,m}=H_{n}(\underbrace{C_{n},\ldots,C_{n}}_{m
\textrm{ times}}),$ where base graph $H=C_{m}$ is a cycle with $m$ vertices and  each of $m$  fibers is $C_{n}.$ 
Here $d_i=2$  and $k_i=1$  for  $i=1,\ldots,m.$
So, the generalized Laplace matrix with respect to the set of
variables $X=(\underbrace{x,\ldots,x}_{m \textrm{ times}})$ has the
form $L(H,X)= \left(\begin{array}{cccccc}
x & -1 & 0 & \ldots & 0& -1 \\
-1 & x & -1 & \ldots & 0& 0 \\
& \vdots &  &  & &\vdots \\
-1 & 0 & 0 & \ldots & -1 & x\\
\end{array}\right).$ Then
$L(H,X)$ is an $m\times m$ circulant matrix with eigenvalues
$\mu_j=x-e^{\frac{2\pi i j}{m}}- e^{-\frac{2\pi i
j}{m}}=x-2\cos(\frac{2\pi j}{m}),j=1,\ldots,m.$ Hence, $\det
L(H,X)=\prod\limits_{j=1}^{m}\mu_j=2 T_{m}(\frac{x}{2})-2.$

We note that $Q_\lambda(w)=\det L(H,W),$  where  $W=(\underbrace{w_1,\ldots,w_m})$  and $w_i=2k_{i}+d_i-\lambda-\sum\limits_{j=1}^{k_i}2\,T_{s_{i,j}}(w)=4-\lambda-2w.$ 
Substituting $x=4-\lambda-2w$  in the formula for $\det L(H,X),$ we get
$Q_\lambda(w)=2 T_{m}(2- w-\frac{\lambda}{2})-2.$ All the roots of the equation $Q_\lambda(w)=0$ is given by the list $w_{j}=2-\frac{\lambda}{2}-\cos(\frac{2\pi j}{m}),\,j=1,2,\ldots,m.$

So, the Laplace characteristic polynomial for discrete torus $T_{n,m}$ is given by
$$\chi_n(\lambda)=\prod\limits_{j=1}^{m}(2T_{n}(w_{j})-2),$$ where $w_{j}=2-\frac{\lambda}{2}-\cos(\frac{2\pi j}{m}),\,j=1,2,\ldots,m.$

\subsection{Cartesian product $H\times C_{n}(s_1,s_2,\ldots,s_{k})$}

Let $H$ be an arbitrary $d$-regular graph on $m$ vertices. The
Cartesian product of a graph $H$ and circulant graph
$G=C_{n}(s_1,s_2,\ldots,s_{k})$ is exactly
$H_{n}(\underbrace{G,\ldots,G}_{m \textrm{ times}}).$ We have $Q_\lambda(w)=\det(L(H,W)),$ where 
$W=(\underbrace{u,u,\ldots,u}_{m  \textrm{ times}})$ and 
$u=2k-\lambda+d-\sum\limits_{j=1}^{k}2T_{s_{j}}(w).$  Note that $L(H,W)=u\,I_{m}-A(H)$, where $A(H)$ be
an adjacency matrix of $H.$ Therefore, 
$Q_\lambda(w)=\chi_{H}(2k-\lambda+d-\sum\limits_{j=1}^{k}2T_{s_{j}}(w)),$ where
$\chi_{H}(x)$ is characteristic polynomial of graph $H.$ In
particular, $Q_\lambda(-1)=\chi_{H}(d+4t-\lambda),$ where $t$ is the number of odd
elements in the set of jumps $\{s_{j},\,j=1,\ldots,k\}.$

\bigskip

\address{
Department of Mathematics \\
Yeungnam University \\
Gyeongsan, Gyeongbuk 38541\\
Korea
}
{ysookwon@ynu.ac.kr}

\address{ 
Sobolev Institute of Mathematics \\
Novosibirsk 630090\\
Novosibirsk State University \\
Novosibirsk 630090\\
Russia
}
{smedn@mail.ru}
%

\address{
Sobolev Institute of Mathematics \\
Novosibirsk 630090\\
Novosibirsk State University \\
Novosibirsk 630090\\
Russia
}
{ilyamednykh@mail.ru}

\end{document}